\documentclass[12pt]{amsart}

\usepackage{amsmath,amssymb,amsthm,mathtools,mathdots,nicefrac,tikz,bbding,stmaryrd,float,multicol,caption,bbm}
\usepackage[mathscr]{euscript}
\usepackage[margin=1in, marginparwidth=1.75cm]{geometry}
\usepackage{color}
\usepackage{hyperref, float}
\hypersetup{colorlinks=true,citecolor=purple,linkcolor=purple}

\theoremstyle{definition}
\newtheorem{theorem}{Theorem}[section]

\newtheorem{lemma}[theorem]{Lemma}
\newtheorem{proposition}[theorem]{Proposition}
\newtheorem{definition}[theorem]{Definition}
\newtheorem{example}[theorem]{Example}
\newtheorem{corollary}[theorem]{Corollary}

\newtheorem{remark}[theorem]{Remark}

\newcommand{\Z}{\mathbb{Z}}
\newcommand{\R}{\mathbb{R}}
\newcommand{\C}{\mathbb{C}}

\newcommand{\Q}{\mathbb{Q}}

\newcommand{\Spec}{\text{Spec}}
\newcommand{\Ass}{\text{Ass}}

\newcommand{\Min}{\text{Min}}
\newcommand{\depth}{\text{depth }}
\newcommand{\height}{\text{ht}}

\newtheorem{manualtheoreminner}{Theorem}
\newenvironment{manualtheorem}[1]{%
  \renewcommand\themanualtheoreminner{#1}%
  \manualtheoreminner
}{\endmanualtheoreminner}

\newcommand{\interior}[1]{%
  {\kern0pt#1}^{\mathrm{o}}%
}

\DeclarePairedDelimiter\abs{\lvert}{\rvert}

\usepackage[backrefs]{amsrefs}

\title{Completions of Uncountable Local Rings with Countable Spectra}
\author{S. Loepp and Teresa Yu}
\begin{document}

\maketitle

\begin{abstract}
We find necessary and sufficient conditions for a complete local (Noetherian) ring to be the completion of an uncountable local (Noetherian) domain with a countable spectrum. Our results suggest that uncountable local domains with countable spectra are more common than previously thought.
We also characterize completions of uncountable excellent local domains with countable spectra assuming the completion contains the rationals, completions of uncountable local unique factorization domains with countable spectra, completions of uncountable noncatenary local domains with countable spectra, and completions of uncountable noncatenary local unique factorization domains with countable spectra.
\end{abstract}

\section{Introduction}

Examples of uncountable Noetherian rings of Krull dimension zero or one that have countable spectra are plentiful. Surprisingly, however, the existence of uncountable Noetherian rings with countable spectra in higher dimensions was not known until 2016, when Colbert constructed in \cite{colbert} an uncountable, $n$-dimensional Noetherian domain with a countable spectrum for any $n\geq 2$. Loepp and Michaelsen extended this result in \cite{loepp2018uncountable} by showing the existence of an uncountable Noetherian domain with a countable spectrum, but with stronger conditions on the ring: they showed, for all $n\geq 0$, the existence of an uncountable, $n$-dimensional, excellent regular local (Noetherian) ring with a countable spectrum.

Given that uncountable Noetherian rings with countable spectra in higher dimensions were not known to exist until very recently, one might expect for such rings to be quite rare. In this paper, we show that these rings are more ``common" than perhaps one might have anticipated, as we find that ``most" complete local rings with a countable residue field are, in fact, the completion of an uncountable local domain with a countable spectrum.  
More precisely, in Section \ref{sec:uncountable}, we prove the following theorem, which characterizes exactly when a complete local ring is the completion of an uncountable domain with a countable spectrum.

\begin{manualtheorem}{\ref{thm:dim2+_precompletion_domain}}
Suppose $T$ is a complete local ring with maximal ideal $M$.
\begin{itemize}
    \item If $\dim T=0$, then $T$ is the completion of an uncountable local domain with a countable spectrum if and only if $T$ is an uncountable field.
    \item If $\dim T=1$, then $T$ is the completion of an uncountable local domain with a countable spectrum if and only if 
\begin{enumerate}
    \item no integer of $T$ is a zero divisor, and
    \item $M\notin\Ass(T)$.
\end{enumerate}
\item If $\dim T\geq 2$, then $T$ is the completion of an uncountable local domain with a countable spectrum if and only if
\begin{enumerate}
    \item no integer of $T$ is a zero divisor,
    \item $M\notin\Ass(T)$, and
    \item $T/M$ is countable.
\end{enumerate}
\end{itemize}
\end{manualtheorem}

Surprisingly, the necessary and sufficient conditions on a complete local ring of dimension at least two to be the completion of an uncountable local domain with a countable spectrum are the same conditions necessary and sufficient for it to be the completion of a countable local domain
\cite[Corollary 3.7]{SMALL19}. 
This is more evidence suggesting that uncountable local domains with countable spectra are more common than expected. 
In order to prove that the conditions of Theorem~\ref{thm:dim2+_precompletion_domain} are sufficient, we construct uncountable local domains from countable local domains while ensuring that the spectra of these rings are order isomorphic when viewed as partially ordered sets.

Natural extensions of this result include characterizing completions of uncountable local rings with countable spectra that satisfy certain properties. In this paper, we provide several extensions of this kind by studying excellent domains, unique factorization domains (UFDs), noncatenary domains, and noncatenary UFDs (Theorems~\ref{thm:dim2+_excellent_precompletion_domain}, \ref{thm:char_ufd_completions_dim_geq2}, \ref{thm:uncount_noncat_domain_char}, and \ref{thm:uncount_noncat_ufd_char}). In general, we are able to construct uncountable local domains with countable spectra and property $X$ by beginning with a countable local domain with property $X$ and constructing from it an uncountable domain whose spectra is order isomorphic to the countable domain's spectra. We then show that the uncountable domain also has property $X$. The countable local domain with property $X$ with which we begin comes from previous results in the literature (for example, \cite[Theorem 8]{heitmannUFD} and results from \cite{paper1}).

The outline of this paper is as follows. In Section~\ref{sec:background}, we provide background. In Section~\ref{sec:construction}, we present a construction of uncountable local domains from countable local domains. Using this construction, we then in Section~\ref{sec:uncountable} characterize completions of uncountable local domains with countable spectra, completions of uncountable excellent local domains with countable spectra such that the completion contains the rationals, completions of uncountable local UFDs with countable spectra, completions of uncountable noncatenary local domains with countable spectra, and completions of uncountable noncatenary local UFDs with countable spectra.

\section{Background}\label{sec:background}

All rings in this paper are commutative with identity. We say a ring is \textit{quasi-local} if it has exactly one maximal ideal but is not necessarily Noetherian, and we say a ring is \textit{local} if it has exactly one maximal ideal and is Noetherian. We denote a quasi-local ring $R$ with unique maximal ideal $M$ by $(R,M)$, and we denote the completion of a local ring $(R,M)$ with respect to its maximal ideal by $\widehat{R}$. 

In \cite{lech}, Lech proves the following result, which characterizes completions of local domains.

\begin{theorem}[\cite{lech}, Theorem 1]\label{thm:lech_char_domain}
A complete local ring $(T,M)$ is the completion of a local domain if and only if
\begin{enumerate}
    \item no integer of $T$ is a zero divisor, and
    \item unless equal to $(0)$, $M\notin\Ass(T)$.
\end{enumerate}
\end{theorem}

A number of more recent results characterize completions of local domains with certain characteristics. Most recently, in \cite{SMALL19}, the authors characterize completions of domains with certain cardinalities and domains with spectra that have certain cardinalities. In particular, they prove the following results.

\begin{theorem}[\cite{SMALL19}, Theorem 2.13]\label{thm:small19_cardinality_domain_char}
Let $(T,M)$ be a complete local ring such that
\begin{enumerate}
    \item no integer of $T$ is a zero divisor, and
    \item unless equal to $(0)$, $M\notin\Ass(T)$.
\end{enumerate}
If $T/M$ is infinite, then $T$ is the completion of a local domain $A$ such that $\abs{A}=\abs{T/M}$. If $T/M$ is finite, then $T$ is the completion of a countable domain.
\end{theorem}

\begin{theorem}[\cite{SMALL19}, Corollary 3.7]\label{thm:small19_domain_char_countable_spec}
Let $(T,M)$ be a complete local ring with $\dim T\geq 2$. Then $T$ is the completion of a local domain with a countable spectrum if and only if
\begin{enumerate}
    \item no integer of $T$ is a zero divisor,
    \item $M\notin\Ass(T)$, and
    \item $T/M$ is countable.
\end{enumerate}
\end{theorem}

Throughout this paper, we show that rings we construct have certain given completions. In order to do this, we make use of the following results, which give sufficient conditions.

\begin{proposition}[\cite{heitmann}, Proposition 1]
If $(R,R\cap M)$ is a quasi-local subring of a complete local ring $(T,M)$, the map $R\to T/M^2$ is onto, and $IT\cap R=IR$ for every finitely generated ideal $I$ of $R$, then $R$ is Noetherian and the natural homomorphism $\widehat{R}=T$ is an isomorphism.
\end{proposition}

\begin{proposition}[\cite{paper1}, 
Corollary 2.5]\label{cor:basering_complete_macine}
Suppose $(T,M)$ is a complete local ring, and $(R,R\cap M)$ is a local subring of $T$ such that $\widehat{R}=T$. Let $(A,A\cap M)$ be a quasi-local subring of $T$ such that $R\subseteq A$, and such that, for every finitely generated ideal $I$ of $A$, $IT\cap A=IA$. Then 
$A$ is Noetherian and $\widehat{A}=T$.
\end{proposition}

Note that if $R$ is a local ring and $\widehat{R}=T$, then $T$ is a faithfully flat extension of $R$. It follows that if $I$ is an ideal of $R$, then $IT\cap R=I$. In addition, $T$ being a faithfully flat extension of $R$ implies that $R$ and $T$ satisfy the going-down theorem, so if $P\in\Spec(T)$, then $\height(P\cap R)\leq\height(P)$. We also have the following relationship between prime ideals of a ring and prime ideals of the ring's completion.

\begin{lemma}[\cite{SMALL19.1}, Lemma 2.6]\label{lem:SMALL19_minprime}
Let $(T,M)$ be the completion of a local ring $(A,A\cap M)$ and let $P$ be a prime ideal of $A$. Then for $Q\in\Min(PT)$, we have $Q\cap A=P$.
\end{lemma}

We end this section with two results on the cardinalities of local rings and their quotient rings.

\begin{proposition}[\cite{paper1}, Proposition 2.10]\label{prop:t/m^2_countable}
Let $(T,M)$ be a local ring. If $T/M$ is finite, then $T/M^2$ is finite. If $T/M$ is infinite, then $\abs{T/M^2}=\abs{T/M}$.
\end{proposition}

Let $c$ denote the cardinality of $\R$.

\begin{lemma}[\cite{dundon}, Lemma 2.2]\label{lem:t_mod_p_uncountable}
Let $(T,M)$ be a complete local ring with $\dim T\geq 1$. Let $P$ be a nonmaximal prime ideal of $T$. Then, $\abs{T/P}=\abs{T}\geq c$.
\end{lemma}

\section{Construction}\label{sec:construction}

Given a countable local domain $(S,S\cap M)$ with completion $(T,M)$, we show in this section how to construct an uncountable local domain $(B,B\cap M)$ such that $S\subseteq B \subseteq T$, ideals of $B$ are extended from $S$, and $\widehat{B}=T$. In the following section, we show how this construction can be used to prove our main results. In particular, we provide specific countable domains from which to begin the construction, and show that the resulting uncountable local domain satisfies certain desirable conditions, including having a countable spectrum.

Before we begin our construction, we introduce some definitions, inspired by definitions from \cite{loepp2018uncountable}.

\begin{definition}\label{defn:s_subring}
Suppose $(T,M)$ is a complete local ring and $(S,S\cap M)$ is a countable local domain such that $S\subseteq T$ and $\widehat{S}=T$. A ring $R$ is an \textit{$S$-subring} of $T$ if $S\subseteq R\subseteq T$, and, for any $r\in R\cap M$, there exist $c\in S\cap M$ and a unit $d\in T$ such that $r=cd$.
\end{definition}

\begin{definition}\label{defn:cs_subring}
Suppose $(T,M)$ is a complete local ring and $(S,S\cap M)$ is a countable local domain such that $S\subseteq T$ and $\widehat{S}=T$. A \textit{$CS$-subring} $R$ of $T$ is a countable quasi-local $S$-subring of $T$ with maximal ideal $R\cap M$.
\end{definition}

Note that, if $(T,M)$ is a complete local ring and $(S,S\cap M)$ is a countable local domain such that $S\subseteq T$ and $\widehat{S}=T$, then $S$ is itself a $CS$-subring of $T$. In addition, the union of an ascending chain of $S$-subrings of $T$ is an $S$-subring of $T$, the countable union of an ascending chain of $CS$-subrings of $T$ is also a $CS$-subring of $T$, and the uncountable union of an ascending chain of $CS$-subrings is a quasi-local $S$-subring of $T$.






The next lemma establishes that $S$-subrings are in fact domains.

\begin{lemma}\label{lem:s_subring_domain}
Suppose $(T,M)$ is a complete local ring and $(S,S\cap M)$ is a countable local domain such that $S\subseteq T$ and $\widehat{S}=T$. Let $R$ be an $S$-subring of $T$. Then, $R$ is an integral domain.
\end{lemma}

\begin{proof}
To prove this, we show that if $P\in\Ass(T)$, then $P\cap R=(0)$. Suppose that $r\in P\cap R \subseteq M\cap R$. Since $R$ is an $S$-subring, there exist $c\in S\cap M$ and $d$, a unit in $T$, such that $r=cd$. Since $d$ is a unit, we have that $c\in P$, and so $c\in P\cap S$. However, $S$ is a domain whose completion is $T$, so $P\cap S=(0)$, and therefore we have that $c=0$. Thus, $r=0$, so $P\cap R=(0)$, and it follows that $R$ is an integral domain.
\end{proof}

Given a countable local domain $(S,S\cap M)$ with completion $(T,M)$, we construct an ascending chain of $CS$-subrings of $T$, starting from $S$, by adjoining elements $u \in T$ of a specific form. Every time we adjoin such a $u$, we will show that we not only obtain another $CS$-subring of $T$, but also that we are indeed adding a new element from $T$. If we do this uncountably many times and then take the union of these rings, we obtain an uncountable $S$-subring of $T$ that we call $(A,A\cap M)$. Finally, we adjoin elements of $T$ to $A$ so that we obtain an $S$-subring of $T$ that is an uncountable local domain with completion $T$ and such that its ideals are extended from $S$.

We begin by describing the elements that we adjoin uncountably many times to the countable local domain $S$. Let $R$ be a $CS$-subring of $T$, and consider $u$ of the form
\[u=1+A_1z_1+A_2z_1z_2+\cdots+A_kz_1z_2\cdots z_k+\cdots,\]
where $(A_i)_{i\in\Z^+}\subseteq R\cap M$ and $(z_i)_{i\in\Z^+}\subseteq S\cap M$. Since $A_i,z_i\in M$ for all $i$, we have that the $k^{\text{th}}$ term in the series is in $M^k$ for $k\geq 2$, so $u$ is indeed an element of $T$. In addition, since every term except for $1$ is an element of $M$ and $1\notin M$, we have that $u$ is a unit in $T$. For $k\in\Z^+$, define
\[M_k\coloneqq 1+A_1 z_1+\cdots +A_{k-1}z_1\cdots z_{k-1}\]
and
\[K_k\coloneqq A_kz_1\cdots z_{k-1}+A_{k+1}z_1\cdots z_{k-1}z_{k+1}+\cdots.\]
Note that we can now express $u$ as $M_k+z_k K_k$ for any $k\in\Z^+$. The next lemma, which is a slight modification of \cite[Lemma 5.1]{loepp2018uncountable}, describes how to adjoin such an element $u$ to a $CS$-subring $R$ of $T$ so that the resulting ring is a $CS$-subring of $T$. We achieve this via an algorithm by choosing the elements of the sequence $(z_i)_{i\in\Z^+}$ given a sequence $(A_i)_{i\in\Z^+}$ satisfying a specific property.

\begin{lemma}\label{lem:z_algo}
Suppose $(T,M)$ is a complete local ring with $\dim T\geq 1$ and $(S,S\cap M)$ is a countable local domain such that $S\subseteq T$ and $\widehat{S}=T$. Let $(R, R \cap M)$ be a $CS$-subring of $T$. Then, for any $(A_i)_{i\in\Z^+}\subseteq R\cap M$ satisfying the property that whenever $i>j$, there exists a $k$ such that $A_i\in M^k$ and $A_j\notin M^k$, there exists a sequence $(z_i)_{i\in\Z^+}\subseteq S\cap M$ with $z_i \neq 0$ for every $i$ such that, if $u\in T$ is of the form
\[u=1+A_1z_1+A_2z_1z_2+\cdots+A_kz_1z_2\cdots z_k+\cdots,\]
then $R[u]_{(R[u]\cap M)}$ is a $CS$-subring of $T$.
\end{lemma}

\begin{proof}
We first use the sequence $(A_i)$ to define the sequence $(z_i)$, so we obtain an element $u$. We then show that adjoining $u$ to $R$ and localizing results in a $CS$-subring of $T$.

Let $R''\coloneqq R[X]$ for an indeterminate $X$. Notice that $R''$ is countable, since $R$ is a $CS$-subring and is thus countable as well. Use the positive integers to enumerate the nonzero elements of $R''$, and consider the $i^{\text{th}}$ element in the well-order, denoted $G_i(X)$. Substituting any $u$ of the form
\[u=1+A_1z_1+\cdots A_2z_1z_2+\cdots+A_kz_1z_2+\cdots z_k+\cdots\]
into the polynomial $G_i(X)$ for any $i\in\Z^+$, we have
\[G_i(u)=r_{\ell,i}u^\ell+\cdots +r_{1,i}u+r_{0,i},\]
where $r_{m,i}\in R$. Since we can express $u$ as $M_k+z_kK_k$ for any $k \geq 1$, we can rewrite $G_i(u)$ as
\[G_i(u)=r_{\ell,i}(M_k+z_kK_k)^\ell +\cdots+ r_{1,i}(M_k+z_kK_k)+r_{0,i}.\]
Expanding this polynomial, we obtain
\begin{align}
    G_i(u)&=r_{\ell,i}\sum_{n=0}^\ell\binom{\ell}{n}M_k^{\ell-n}z_k^n K_k^n+\cdots+r_{2,i}\sum_{n=0}^2\binom{2}{n}M_k^{2-n}z_k^n K_k^n+r_{1,i}(M_k+z_kK_k)+r_{0,i}\nonumber\\
    &=r_{\ell,i}\left(M_k^\ell+\sum_{n=1}^\ell M_k^{\ell-n}z_k^n K_k^n\right)+\cdots+r_{1,i}(M_k+z_kK_k)+r_{0,i}\nonumber\\
    &=G_i(M_k)+r_{\ell,i}\sum_{n=1}^\ell M_k^{\ell-n}z_k^n K_k^n+\cdots+r_{1,i}z_kK_k\nonumber\\
    &=G_i(M_k)+z_k\left(\sum_{m=1}^\ell r_{m,i}\sum_{j=1}^m\binom{m}{j}z_{k}^{j-1}K_k^j M_k^{m-j}\right). \label{eqn:domain_binom}
\end{align}

We will define the sequence $(z_i)_{i\in\Z^+}$ recursively using $G_j(M_i)$ so that $z_i\in S\cap M$ for every $i\in\Z^+$. Since we define each $z_i$ to be in $S\cap M$, the second of the two terms in (\ref{eqn:domain_binom}) is an element of $M$. 
Thus, for any $k\geq 1$, $G_i(u)\in M$ if and only if $G_i(M_k)\in M$. 
In other words, for any $k\geq 1$, $G_i(u)$ is a unit in $T$ if and only if $G_i(M_k)$ is a unit in $T$. 

By hypothesis, each $A_i\in R\cap M$, so since each $M_i$ is defined to be $M_i=1+A_1z_1+\cdots+A_{i-1}z_1\cdots z_{i-1}$ and each term of this sum is in $R$, we have that $M_i\in R$. Thus, since $G_j(X)\in R[X]$, we see that $G_j(M_i)\in R$ for all $i,j\in\Z^+$.

Notice that $S\cap M\neq (0)$, since $\widehat{S}=T$ and $\dim T\geq 1$. Let $x$ denote a non-zero element of $S\cap M$.

Starting with $j=1$ and $i=1$, we use $G_j(M_i)$ to define $z_i$ as follows:
\begin{enumerate}
    \item If $G_j(M_i)$ is a unit, then let $z_i=x$, and use $G_{j+1}(M_{i+1})$ to define $z_{i+1}$.
    \item If $G_j(M_i)=0$, then let $z_i=x$, and use $G_j(M_{i+1})$ to define $z_{i+1}$.
    \item If $G_j(M_i)$ is nonzero and not a unit, then since $G_j(M_i)\in R$, which is a $CS$-subring, there exist $c\in S\cap M$ and $d$ a unit in $T$ such that $G_j(M_i)=cd$. Since $G_j(M_i)\neq 0$, we have that $c\neq 0$. Let $z_i=c$, and use $G_{j+1}(M_{i+1})$ to define $z_{i+1}$.
\end{enumerate}
Notice that, by this definition, $z_i\neq 0$ and $z_i\in S\cap M$ for all $i$.

Using our given $(A_i)$ sequence and this definition for the sequence $(z_i)$, define the element $u$ in the previously specified form, and let $R'\coloneqq R[u]$. We show that $R'_{(R'\cap M)}$ is a $CS$-subring of $T$. Notice that, since $R'$ is countable and $S\subseteq R\subseteq R'$,
 it is enough to show that if $r\in R'\cap M$, then $r$ is of the form $r=cd$, where $c\in S\cap M$ and $d$ is a unit in $T$.  If $r = 0$, then $c=0$, $d=1$ works.

Suppose that $r\in R'\cap M$ is a nonzero element of $R'$, so $r=G_i(u)$ for some $i\in\Z^+$ and $r$ is not a unit. Then $G_i(M_k)$ is not a unit for all $k\in\Z^+$ by the claim above. We show that $G_i(u)=cd$ for $c\in S\cap M$ and $d\in T$ a unit. In order to do so, we first need to show that every $M_i$ is distinct. Suppose not, and that $M_j=M_i$ for some $i>j$. Then,
\[M_i-M_j=A_jz_1\cdots z_j+A_{j+1}z_1\cdots z_{j+1}+\cdots +A_{i-1}z_1\cdots z_{i-1}=0.\]
Since every $z_i$ is nonzero and $R$ is a domain, we can cancel $z_1\cdots z_j$, obtaining
\[A_j+A_{j+1}z_{j+1}+\cdots+A_{i-1}z_{j+1}\cdots z_{i-1}=0.\]
By our assumption on the sequence $(A_i)$, there exists a $k$ such that $A_{j+1}\in M^k$ but $A_j\notin M^k$. Then, subtracting $A_j$ from both sides above, we see that
\[A_{j+1}z_{j+1}+\cdots+A_{i-1}z_{j+1}\cdots z_{i-1}=-A_j\notin M^k.\]
However, $A_{j+1}\in M^k$, so we have that $A_\ell\in M^k$ for every $\ell\geq j+1$; thus,
\[A_{j+1}z_{j+1}+\cdots A_{i-1}z_{j+1}z_{j+2}\cdots +z_{i-1}\in M^k,\]
which is a contradiction. Thus, all of the $M_i$'s are distinct.

Since $0\neq G_i(X)\in R[X]$ and $R$ is an integral domain, it must be that $G_i(X)$ has at most $\text{deg}(G_i)$ roots in $R$. Each $M_i\in R$ and all of the $M_i$'s are distinct, so there are infinitely many distinct $M_i$'s, not all of which can be one of the finitely many roots of $G_i(X)$. Thus, there exists some $k\in\Z^+$ such that $G_i(M_k)\neq 0$. By case (3) in the algorithm above, we have that $G_i(M_k)=z_kd'$, where $d'$ is a unit. Substituting into (\ref{eqn:domain_binom}), we have
\begin{align}
    G_i(u)=z_k\left(d'+\sum_{m=1}^\ell r_{m,i}\sum_{j=1}^m\binom{m}{j}z_{k}^{j-1}K_k^j M_k^{m-j}\right).\label{eqn:domain_binom_substitution}
\end{align}
Notice that $z_k\in S\cap M$ by definition. In addition, recall that $K_k$ is defined to be in $M$ as well, so since $d'$ is a unit in $T$, the element
\[d'+\sum_{m=1}^\ell r_{m,i}\sum_{j=1}^m\binom{m}{j}z_{k}^{j-1}K_k^j M_k^{m-j}\]
is a unit in $T$. Thus, $r=G_i(u)$ can be written as $cd$, where $c\in S\cap M$ and $d\in T$ is a unit. 

It follows that $R'$ localized at $R'\cap M$ is a $CS$-subring of $T$.
\end{proof}

The next lemma shows that there always exists a choice of $(A_i)_{i\in\Z^+}\subseteq R\cap M$ such that the ring $R[u]_{R[u]\cap M}$ from Lemma~\ref{lem:z_algo} is not equal to $R$. The statement and proof of this lemma are very similar to those of \cite[Lemma 5.2]{loepp2018uncountable}

\begin{lemma}\label{lem:strict_containment}
Suppose $(T,M)$ is a complete local ring with $\dim T\geq 1$ and $(S,S\cap M)$ is a countable local domain such that $S\subseteq T$ and $\widehat{S}=T$. Given a $CS$-subring $(R,R \cap M)$ of $T$, there exists a $CS$-subring $(R',R' \cap M)$ of $T$ where $R\subsetneq R'\subseteq T$.
\end{lemma}

\begin{proof}

Since $\dim T \geq 1$ and $\widehat{S} = T$, we have that $\dim S \geq 1$ and so $S\cap M\neq(0)$.  Let $x \in S\cap M \subseteq R \cap M$ with $x \neq 0$. 
Define $A_i=x^{q(i)}$, where $q:\Z^+\to\Z^+$ is a strictly increasing function. Notice that $(A_i)_{i\in\Z^+}\subseteq R\cap M$.  Suppose that $i$ and $j$ are positive integers such that $j < i$.  Then $q(j) < q(i)$.   Since $x \neq 0$ and $S$ is a domain, $x^s \neq 0$ for all positive integers $s$.  Let $\ell$ be the largest integer such that $x^{q(j)} \in M^\ell$, and note that $\ell$ exists because $\bigcap_{i=1}^\infty M^i=(0)$.  We then have that 
$x^{q(j)} \not\in M^{\ell+1}$, but $x^{q(j) + 1} \in M^{\ell+1}$.  As $q(i) \geq q(j) + 1$, we have that
$x^{q(i)} \in M^{\ell+1}$.
This shows that our sequence of $A_i$'s satisfies the needed conditions for Lemma~\ref{lem:z_algo}: we have that $(A_i)_{i\in\Z^+}\subseteq R\cap M$, and whenever $i>j$, there exists a $k\in \Z^+$ such that $A_i\in M^k$ but $A_j\notin M^k$. Thus, by Lemma~\ref{lem:z_algo}, there exists a sequence $(z_i)_{i\in\Z^+}\subseteq S\cap M$ with $z_i \neq 0$ for every $i$ such that, if $u\in T$ is of the form
\[u=1+A_1z_1+A_2z_1z_2+\cdots+A_kz_1z_2\cdots z_k+\cdots,\]
then $R[u]_{R[u]\cap M}$ is a $CS$-subring of $T$.

We now show that there are uncountably many choices for the element $u$. First, notice that, by a diagonal argument, there are uncountably many choices for the function $q$. We show that distinct choices for the function $q$ yield distinct $u$'s. 

Let $u_1=1+A_1z_1+\cdots$ and $u_2=1+B_1z'_1+\cdots$, where $A_i=x^{q(i)}$ and $B_i=x^{p(i)}$, with $p,q:\Z^+\to\Z^+$ both strictly increasing. Suppose that $u_1=u_2$. We show that $A_i=B_i$ and $z_i=z'_i$ for all $i$.

First, referring back to the algorithm described in the proof of Lemma~\ref{lem:z_algo}, $M_1=1$ for all choices of $u$. Since $z_1$ and $z'_1$ are both defined by the algorithm using $G_1(M_1)=G_1(1)$, they are the same and both nonzero. Then, equating $u_1$ and $u_2$, we have $A_1+A_2z_2+\cdots=B_1+B_2z'_2+\cdots$, so then
\[x^{q(1)}+x^{q(2)}z_2+\cdots=x^{p(1)}+x^{p(2)}z'_2+\cdots.\]
Without loss of generality, suppose that $q(1)\leq p(1)$; then, cancelling by $x^{q(1)}$, we have
\[1+x^{q(2)-q(1)}z_2+\cdots=x^{p(1)-q(1)}+x^{p(2)-q(1)}z'_2+\cdots.\]
Since $z_i\in S\cap M$ for all $i$ and $1\notin M$, the left-hand side is not in $M$. Thus, the the right-hand side is not in $M$ either. However, all but the first term are in $M$, since $z'_i\in S\cap M$. Thus, $x^{p(1)-q(1)}\notin M$, implying that $p(1)-q(1)=0$, so $p(1)=q(1)$; thus, $A_1=B_1$. For the algorithm in the proof of Lemma ~\ref{lem:z_algo} the definitions of $z_i$ (resp., $z'_i$) depend only on $A_j$ and $z_j$ (resp., $B_j$ and $z'_j$) for all $j<i$.  Since we have shown that $z_1=z'_1$ and $A_1=B_1$, we can show inductively that $z_i=z'_i$ and $A_i=B_i$ for all $i\geq 1$. There are uncountably many choices for the function $q$, so there are uncountably many choices for $u$ such that $R[u]_{R[u]\cap M}$ is a $CS$-subring of $T$. Since $R$ is a $CS$-subring of $T$, it is countable, and thus there exists a $u\in T\setminus R$ such that $R[u]_{R[u]\cap M}$ is a $CS$-subring of $T$. Hence, $R'=R[u]_{R[u]\cap M}$ is the desired $CS$-subring of $T$.
\end{proof}

The next theorem, which is a generalization of \cite[Theorem 5.3]{loepp2018uncountable}, guarantees the existence of an uncountable quasi-local $S$-subring of $T$.

\begin{theorem}\label{thm:existence_uncountable_s*}
Suppose $(T,M)$ is a complete local ring with $\dim T\geq 1$ and $(S,S\cap M)$ is a countable local domain such that $S\subseteq T$ and $\widehat{S}=T$.
There exists an uncountable quasi-local $S$-subring of $T$, denoted $(A,A\cap M)$.
\end{theorem}

\begin{proof}
There exists a well-ordered uncountable set $C$ such that every element of $C$ has only countably many predecessors.  Let $0$ denote the minimal element of $C$. For every element $c\in C$, we inductively define a $CS$-subring of $T$ that we denote $S_c$.

 First, let $S_0\coloneqq S$. Suppose $0<c\in C$, and assume that $S_b$ has been defined for every $b<c$ so that $S_b$ is a $CS$-subring of $T$. If $c$ has a predecessor, $b\in C$, then define $(S_c, S_c \cap M)$ to be the $CS$-subring of $T$ obtained from Lemma~\ref{lem:strict_containment} with $R=S_b$ so that $S_b\subsetneq S_c\subseteq T$. If instead $c$ is a limit ordinal, define $S_c=\bigcup_{b<c}S_b$; since $c$ has countably many predecessors $b<c$ and each $S_b$ is a $CS$-subring of $T$, we have that $S_c$ is a $CS$-subring of $T$. Then $S_c$ is a $CS$-subring of $T$ for every $c \in C$.

Let $A\coloneqq\bigcup_{c\in C}S_c$. Since each $S_c$ is a $CS$-subring of $T$, $A$ is a quasi-local $S$-subring of $T$ with unique maximal ideal $A\cap M$. Finally, notice that, for uncountably many $c\in C$, there exists a predecessor $b$ of $c$ such that $S_b\subsetneq S_c$, so then $A$ is uncountable. Thus, $(A,A\cap M)$ is an uncountable quasi-local $S$-subring of $T$.
\end{proof}

Next, we construct an uncountable quasi-local $S$-subring $(B,B\cap M)$ of $T$ so that $bT\cap B=bB$ for all $b\in B$. This construction is inspired by the construction from \cite[Theorem 5.4]{loepp2018uncountable}.

\begin{theorem}\label{thm:existence_uncountable_idealprop_s*}
Suppose $(T,M)$ is a complete local ring with $\dim T\geq 1$ and $(S,S\cap M)$ is a countable local domain such that $S\subseteq T$ and $\widehat{S}=T$.
There exists an uncountable quasi-local $S$-subring of $T$, $(B,B\cap M)$, such that, for every principal ideal $I$ of $B$, $IT\cap B=IB$. 
\end{theorem}

\begin{proof}
Define $B=QF(A)\cap T$, where $(A, A \cap M)$ is the uncountable quasi-local $S$-subring of $T$ obtained from Theorem~\ref{thm:existence_uncountable_s*}, and $QF(A)$ is the quotient field of $A$. Then, $A\subseteq B$, and since $A$ is uncountable, so is $B$. Let $r\in B\cap M$, so $r=\frac{a}{b}$ for $a,b\in A$ with $b\neq 0$. We show that $r=cd$, where $c\in S\cap M$ and $d\in T$ is a unit. Since $\frac{a}{b}\in M$, there exists $m\in M$ such that $\frac{a}{b}=m$, so $a=bm$, implying that $a\in A\cap M$. Recall that $A$ is an $S$-subring, so there exist $c\in S\cap M$ and $d\in T$ a unit such that $a=cd$. Thus, $\frac{a}{b}=\frac{cd}{b}$. First suppose that $b\in A\setminus M$; then $b$ is a unit in $T$, so $r=c(db^{-1})$ is of the desired form, since $db^{-1}\in T$ is a unit. 

Now suppose that $b\in A\cap M$, so $b=c'd'$ for $c'\in S\cap M$ and $d'\in T$ a unit. Notice that $c'\neq 0$ since $b\neq 0$. Then,
\[\frac{a}{b}=\frac{cd}{c'd'}\in B\cap M\subseteq T,\]
so, for some $v\in T$, we have that $\frac{cd}{c'd'}=v$, implying that $c=c'd'd^{-1}v$. This means that $c\in c'T\cap S$. Recall that the completion of $S$ is $T$, and so $c'T\cap S=c' S$; in particular, $c=c's$ for some $s\in S$. Thus,
\[\frac{a}{b}=\frac{cd}{c'd'}=\frac{c'sd}{c'd'}=\frac{sd}{d'}=s(dd'^{-1}).\]
If $s\in S\setminus M$, then $\frac{a}{b}\notin M$, which contradicts the assumption that $\frac{a}{b}\in B\cap M$. It must be that $s\in S\cap M$, and so $r=\frac{a}{b}$ can be written in the desired form. Thus, we have that $B$ is indeed an $S$-subring of $T$.

We show that $B$ is quasi-local with maximal ideal $B\cap M$ by showing that the units in $B$ are exactly the elements that are not in $B\cap M$. If $x\in B$ is a unit in $B$, then $x$ is a unit in $T$, so $x\notin M$ and $x\notin B\cap M$. Now suppose that $x\in B\setminus (B\cap M)$. Then, $x\neq 0$ and $x\notin M$, so $x$ must be a unit in $T$. There therefore exists a $y\in T$ with $xy=1$. Since $y=\frac{1}{x}\in QF(A)\cap T=B$, we have that $x,y\in B$ and $x$ is a unit in $B$.

Finally, we show that if $I$ is a principal ideal of $B$ then $IT\cap B=IB$. Notice that $I=\left(\frac{a}{b}\right)B$ for some $a,b\in A$ with $b$ nonzero. If $a=0$, then, $IT=(0)$ and $IB=(0)$, so $IT\cap B=(0)=IB$. Now suppose that $a\neq 0$. If $c\in IT\cap B$, then $c=\frac{a}{b}t$ for some $t \in T$. In addition, $c=\frac{a'}{b'}$ for some $a',b'\in A$ with $b'\neq 0$, since $c\in B$. Now, $t=\frac{a'b}{b'a}\in B$, so $c\in\frac{a}{b}B=IB$. It follows that $IT\cap B\subseteq IB$. Since $IB\subseteq IT$ and $IB\subseteq B$, we have that $IB\subseteq IT\cap B$. Thus, $IT\cap B=IB$, as desired.
\end{proof}

We now use the properties of the ring $B$ constructed in Theorem~\ref{thm:existence_uncountable_idealprop_s*} to show, in the next three results, that $B$ Noetherian and has completion $T$. These results are generalizations of \cite[Lemma 6.1, Theorem 6.2, Theorem 6.3]{loepp2018uncountable}.

\begin{lemma}\label{lem:fgideals_extended}
Suppose $(T,M)$ is a complete local ring with $\dim T\geq 1$ and $(S,S\cap M)$ is a countable local domain such that $S\subseteq T$ and $\widehat{S}=T$. 
Then there exists an uncountable quasi-local $S$-subring of $T$, $(B,B\cap M)$, such that finitely generated ideals of $B$ are extended from $S$, i.e., for any finitely generated ideal $J$ of $B$, $J=(p_1,\ldots,p_k)B$ for $p_i\in S$.
\end{lemma}

\begin{proof}
Let $(B,B\cap M)$ be the uncountable quasi-local $S$-subring of $T$ constructed in Theorem~\ref{thm:existence_uncountable_idealprop_s*}, such that $IT\cap B=IB$ for every principal ideal $I$ of $B$.

If $J$ is a finitely generated ideal of $B$, then $J=(b_1,\ldots,b_k)B$ for some $b_i\in B$. If $J = (0)$, then $J$ is extended from the zero ideal of $S$.  If $J = B$ then $J=1B$ and $1\in S$.  So assume that $b_i$ are nonzero nonunits for all $i = 1,2, \ldots, k$.

Then, we have that $b_i\in B\cap M$, so $b_i=p_iu_i$, where $p_i\in S\cap M$, and where $u_i$ is a unit in $T$ for all $i=1,\ldots,k$. We show that each $u_i$ is also a unit in $B$. Notice that $b_i=p_iu_i\in p_i T\cap B$. Since $p_iB$ is a principal ideal of $B$, we have that $p_iT\cap B=p_iB$ by Theorem~\ref{thm:existence_uncountable_idealprop_s*}. Thus, $p_iu_i\in p_i B$, and since $p_i$ is not a zero divisor in $T$, we have that $u_i\in B$. Since $u_i\notin M$, it must be that $u_i$ is a unit in $B$. Then, the $p_i$'s and $b_i$'s generate the same ideals in $B$ since they are associates, so $J=(b_1,\ldots,b_k)B=(p_1,\ldots,p_k)B$. Recall that $p_i\in S$ for all $i$, so $J$ is extended from $S$.
\end{proof}

\begin{theorem}\label{thm:b_idealprop}
Suppose $(T,M)$ is a complete local ring with $\dim T\geq 1$ and $(S,S\cap M)$ is a countable local domain such that $S\subseteq T$ and $\widehat{S}=T$. 
Then there exists an uncountable quasi-local $S$-subring of $T$, $(B,B\cap M)$, such that, for every finitely generated ideal $I$ of $B$, $IT\cap B=IB$, and finitely generated ideals of $B$ are extended from $S$.
\end{theorem}

\begin{proof}
Let $(B,B\cap M)$ be the uncountable quasi-local $S$-subring of $T$ whose existence is guaranteed by Lemma~\ref{lem:fgideals_extended}. We have that every finitely generated ideal of $B$ is extended from $S$. In addition, we always have that $IB\subseteq IT\cap B$, so we now show that $IT\cap B\subseteq IB$. 

Let $I$ be a finitely generated ideal of $B$, which can be written as $I=(p_1,\ldots,p_k)B$ for some $p_i\in S$ by Lemma~\ref{lem:fgideals_extended}. Consider $c\in IT\cap B$. We will show that $c\in IB$. First, if $I=B$, then $IT\cap B=BT\cap B=B=IB$, so now suppose $I\neq B$. Then, $IT\subseteq M$, so $c\in IT\cap B\subseteq M\cap B$. Note that $c=qu$ for some $q\in S\cap M$ and unit $u\in T$ since $B$ is an $S$-subring of $T$. Then, $qu=c\in qT\cap B=qB$. But this implies that $u\in B$ and it is a unit in $B$. Finally, observe that $cu^{-1}\in (p_1,\ldots,p_k)T$, since $c\in (p_1,\ldots,p_k)T\cap B$ and $u^{-1}\in T$, and recall that $cu^{-1}=q\in S$. Also recall that $\widehat{S}=T$, and so for any finitely generated ideal $I$ of $S$, we have that $IT\cap S=IS$. Thus,
\[cu^{-1}=q\in (p_1,\ldots,p_k)T\cap S=(p_1,\ldots,p_k)S\subseteq (p_1,\ldots,p_k)B=IB.\]
Since $u$ is a unit in $B$, we have that $c\in IB$.
\end{proof}

\begin{theorem}\label{thm:B_completionT}
Suppose $(T,M)$ is a complete local ring with $\dim T\geq 1$ and $(S,S\cap M)$ is a countable local domain such that $S\subseteq T$ and $\widehat{S}=T$. Then there exists an uncountable local domain $(B,B\cap M)$ such that $S\subseteq B\subseteq T$, $\widehat{B}=T$, and ideals of $B$ are extended from $S$.
\end{theorem}

\begin{proof}
By Theorem~\ref{thm:b_idealprop}, there exists an uncountable quasi-local $S$-subring of $T$, $(B,B\cap M)$, such that $IT\cap B=IB$ for all finitely generated ideals $I$ of $B$, and all finitely generated ideals of $B$ are extended from $S$.

By Lemma~\ref{lem:s_subring_domain}, we have that $B$ is a domain because it is an $S$-subring of $T$. 
In addition, $S\subseteq B$ and $(S,S\cap M)$ is a local subring of $T$ such that $\widehat{S}=T$. Thus, by Proposition~\ref{cor:basering_complete_macine}, $B$ is Noetherian with completion $T$.  Since $B$ is Noetherian, all of its ideals are finitely generated, and it follows that all ideals of $B$ are extended from $S$.
\end{proof}

\section{Main Results}\label{sec:uncountable}

In this section, we use our results from Section~\ref{sec:construction} to characterize completions of uncountable local domains with countable spectra that satisfy various properties, such as being excellent, a UFD, and noncatenary.

In order to show many of our results, we view the spectra of the rings in question as \textit{partially ordered sets (posets)} under the relation of set containment. For two posets, $(X,\leq_X)$ and $(Y,\leq_Y)$, an \textit{order isomorphism} from $X$ to $Y$ is a bijective function $f:X\to Y$ such that for all $x,x'\in X$, we have that $x\leq_X x'$ if and only if $f(x)\leq_Y f(x')$. We now show that if two local domains, one contained in the other, have the same completion and satisfy the condition that every ideal of the larger one is extended from the smaller one, then the two local domains' spectra are order isomorphic.

\begin{proposition}\label{prop:specb_specs_isomorphic}
Let $(T,M)$ be a complete local ring, and suppose that $(S,S\cap M)$ and $(B,B\cap M)$ are local domains such that $S\subseteq B$, $\widehat{S}=\widehat{B}=T$, and every ideal of $B$ is extended from $S$. Then the mapping $\varphi:\Spec(B)\to\Spec(S)$ given by $P\mapsto P\cap S$ is an order isomorphism.
\end{proposition}

\begin{proof} Note that if $P$ is a prime ideal of $B$, then $P \cap S$ is a prime ideal of $S$.

We show that $\varphi$ is an order isomorphism by showing that $\varphi$ is surjective, and that if $P,Q\in\Spec(B)$, then $P\subseteq Q$ if and only if $\varphi(P)\subseteq\varphi(Q)$. Notice that this latter condition also shows injectivity: if $\varphi(P)=\varphi(Q)$, then we have that $\varphi(P)\subseteq\varphi(Q)$ and $\varphi(Q)\subseteq\varphi(P)$, thus ensuring that $P\subseteq Q$ and $Q\subseteq P$, i.e., $P=Q$.

To show that $\varphi$ is surjective, suppose $Q\in\Spec(S)$, and let $P$ be a prime ideal of $T$ with $P\in\Min(QT)$. Then, $P\cap B\in\Spec(B)$. Furthermore, $\varphi(P\cap B)=(P\cap B)\cap S=P\cap S=Q$, by Lemma~\ref{lem:SMALL19_minprime}. Thus, $\varphi$ is surjective.

Suppose $P,Q\in\Spec(B)$ with $P\subseteq Q$. Then $P\cap S\subseteq Q\cap S$, and so $\varphi(P)\subseteq\varphi(Q)$. Now suppose $\varphi(P)\subseteq\varphi(Q)$, with $P=(x_1,\ldots,x_n)B$ for $x_i\in S$. 
Then, for all $i = 1,2, \ldots n$, we have that $x_i \in P \cap S = \varphi(P) \subseteq \varphi(Q) = Q \cap S$.  It follows that $x_i \in Q$ for all $i$, and we have that $P \subseteq Q$.
Thus, we have that $\varphi$ is an order isomorphism and that $\Spec(B)$ and $\Spec(S)$ are isomorphic as posets.
\end{proof}

If the spectra of two rings are isomorphic as posets, then, for any prime ideal of height $k$ in one ring, there exists a corresponding prime ideal of height $k$ in the other ring. Similarly, if there is a saturated chain of prime ideals of length $n$ between two prime ideals $P$ and $Q$ of one ring, then there exist corresponding prime ideals $P'$ and $Q'$ in the other ring, as well as a saturated chain of prime ideals of length $n$ between them.

\begin{remark}\label{rem:BS_same_generators}
Using the isomorphism from Proposition~\ref{prop:specb_specs_isomorphic}, we note that the prime ideals of $B$ and $S$ that correspond to each other are generated by the same elements of $S$. If $P\in\Spec(B)$, then $P=(a_1,\ldots,a_n)B$ with $a_i\in S$ for all $i$. Then, since $\widehat{B}=\widehat{S}=T$ and since $S\subseteq B$, we have the following equalities:
\[\varphi(P)=(a_1,\ldots,a_n)B\cap S=((a_1,\ldots,a_n)T\cap B)\cap S=(a_1,\ldots,a_n)T\cap S=(a_1,\ldots,a_n)S.\]
\end{remark}

We now focus on characterizing completions of uncountable local domains with countable spectra. In the dimension zero case, this characterization comes as a direct consequence of the fact that the completions of fields are themselves, and the fact that all local domains of dimension zero have exactly one prime ideal.

\begin{proposition}\label{prop:dim0_precompletion_domain}
Suppose $(T,M)$ is a complete local ring of dimension zero. Then $T$ is the completion of an uncountable local domain with a countable spectrum if and only if $T$ is an uncountable field.
\end{proposition}

\begin{proof}
First suppose $T$ is the completion of an uncountable local domain with a countable spectrum. This uncountable local domain with a countable spectrum must be dimension zero, so it must be a field. However, the completion of a field is itself, so it must be $T$. Thus, $T$ is uncountable and a field.

Now suppose $T$ is an uncountable field. Then, $T$ has one prime ideal and $\widehat{T}=T$. Thus, $T$ is the completion of an uncountable local domain with a countable spectrum.
\end{proof}

We use our construction from Section~\ref{sec:construction} to tackle the sufficient conditions in the case that the dimension of the rings in question are at least one. First, we show that the uncountable local domain that we construct in the previous section has a countable spectrum.

\begin{proposition}\label{prop:B_countable_spec}
Let $(T,M)$ be a complete local ring, and suppose that $(S,S\cap M)$ and $(B,B\cap M)$ are local domains such that $S$ is countable, $S\subseteq B$, $\widehat{S}=\widehat{B}=T$, and every ideal of $B$ is extended from $S$. Then $\Spec(B)$ is countable.
\end{proposition}

\begin{proof}
By Proposition~\ref{prop:specb_specs_isomorphic}, we have that $\Spec(B)$ and $\Spec(S)$ are isomorphic as posets. In particular, they have the same cardinalities, so since $\Spec(S)$ is countable, $\Spec(B)$ is countable as well.
\end{proof}

The following result from \cite{SMALL19} characterizes completions of countable local domains.

\begin{theorem}[\cite{SMALL19}, Corollary 2.15]\label{prop:small19_countable_domain}
Let $(T,M)$ be a complete local ring. Then $T$ is the completion of a countable local domain if and only if
\begin{enumerate}
    \item no integer is a zero divisor of $T$,
    \item unless equal to $(0)$, $M\notin\Ass (T)$, and
    \item $T/M$ is countable
\end{enumerate}
\end{theorem}

We use Theorem \ref{prop:small19_countable_domain} to identify sufficient conditions for a complete local ring of dimension at least one to be the completion of an uncountable local domain with a countable spectrum.

\begin{proposition}\label{prop:uncountable_domain_countable_spec_sufficient}
Suppose $(T,M)$ is a complete local ring with $\dim T\geq 1$ and suppose the following conditions are satisfied:
\begin{enumerate}
    \item no integer of $T$ is a zero divisor,
    \item $M\notin\Ass(T)$, and 
    \item $T/M$ is countable.
\end{enumerate}
Then $T$ is the completion of an uncountable local domain with a countable spectrum.
\end{proposition}

\begin{proof}
If $T$ satisfies conditions (1), (2), and (3), then, by Theorem~\ref{prop:small19_countable_domain}, $T$ is the completion of a countable local domain. Let this countable local domain be $(S,S\cap M)$ and apply Theorem~\ref{thm:B_completionT}. Then, we have that there exists an uncountable local domain $(B,B\cap M)$ such that $S\subseteq B\subseteq T$, $\widehat{B}=T$, and ideals of $B$ are extended from $S$. Finally, by Proposition~\ref{prop:B_countable_spec}, we have that $B$ has a countable spectrum.
\end{proof}

We are now able to prove one of our main results.

\begin{theorem}\label{thm:dim2+_precompletion_domain}
Suppose $(T,M)$ is a complete local ring.
\begin{itemize}
    \item If $\dim T=0$, then $T$ is the completion of an uncountable local domain with a countable spectrum if and only if $T$ is an uncountable field.
    \item If $\dim T=1$, then $T$ is the completion of an uncountable local domain with a countable spectrum if and only if 
\begin{enumerate}
    \item no integer of $T$ is a zero divisor, and
    \item $M\notin\Ass(T)$.
\end{enumerate}
    \item If $\dim T\geq2$, then $T$ is the completion of an uncountable local domain with a countable spectrum if and only if
\begin{enumerate}
    \item no integer of $T$ is a zero divisor,
    \item $M\notin\Ass(T)$, and
    \item $T/M$ is countable. 
\end{enumerate}
\end{itemize}
\end{theorem}

\begin{proof}
First, if $\dim T=0$, then the desired statement is given by Proposition~\ref{prop:dim0_precompletion_domain}.

Suppose that $(T,M)$ is a complete local ring with $\dim T\geq 1$, no integer of $T$ is a zero divisor, $M\notin\Ass(T)$, and $T/M$ is countable. By Proposition~\ref{prop:uncountable_domain_countable_spec_sufficient} we have that $T$ is the completion of an uncountable local domain with a countable spectrum. Now suppose that $(T,M)$ is a complete local ring with $\dim T=1$, no integer of $T$ is a zero divisor, $M\notin\Ass(T)$, and $T/M$ is uncountable. By Theorem~\ref{thm:small19_cardinality_domain_char}, $T$ is the completion of an uncountable local domain $A$. Since $\dim A=1$, $A$ has a countable spectrum.

Now suppose that $(T,M)$ is a complete local ring with $\dim T\geq 1$ and it is the completion of an uncountable local domain with a countable spectrum. By Theorem~\ref{thm:lech_char_domain}, it must be that no integer of $T$ is a zero divisor and $M\notin\Ass(T)$. If $\dim T\geq 2$, then it must also be that $T/M$ is countable by Theorem~\ref{thm:small19_domain_char_countable_spec}.
\end{proof}

We present a corollary of our result in the case that $T$ itself is a domain.

\begin{corollary}\label{cor:T_domain}
Suppose $(T,M)$ is a complete local domain. 
\begin{itemize}
    \item If $\dim T=0$, then $T$ is the completion of an uncountable local domain with a countable spectrum if and only if $T$ is uncountable. 
    \item If $\dim T=1$, then $T$ is the completion of an uncountable local domain with a countable spectrum. 
    \item If $\dim T\geq 2$, then $T$ is the completion of an uncountable local domain with a countable spectrum if and only if $T/M$ is countable.
\end{itemize}
\end{corollary}

\begin{proof}
Suppose $\dim T = 0$.  Since $T$ is a domain, it is a field, and so the result follows from Theorem \ref{thm:dim2+_precompletion_domain}.
Now suppose $\dim T \geq 1$.  Since $T$ is a domain, we have that no integer of $T$ is a zero divisor and $M \notin\Ass(T)$.  If $\dim T = 1$, then $T$ is the completion of an uncountable local domain with a countable spectrum by Theorem \ref{thm:dim2+_precompletion_domain}.
Suppose that $\dim T \geq 2$.  Then, by Theorem \ref{thm:dim2+_precompletion_domain}, $T$ is the completion of an uncountable local domain with a countable spectrum if and only if $T/M$ is countable.    
\end{proof}

\begin{example}\label{eg:uncountable_domain_completion}
By Theorem~\ref{thm:dim2+_precompletion_domain}, we have that any complete local ring of dimension at least $2$ that is of the form $\R[[x_1,\ldots,x_n]]/I$ or $\C[[x_1,\ldots,x_n]]/I$ is not the completion of an uncountable local domain with a countable spectrum.

Using Corollary~\ref{cor:T_domain}, we see that the complete local ring $T=\Q[[x_1,\ldots,x_n]]$ for any $n\geq 2$ is the completion of an uncountable local domain with a countable spectrum, as $T/M=\Q$ is countable. Note that $T$ is one of the ``nicest'' examples of a complete local ring with a countable residue field, so if any ring were to be the completion of an uncountable local domain with a countable spectrum, it is not that surprising that $T$ would be one. However, we know by Theorem~\ref{thm:dim2+_precompletion_domain} that the complete local ring $T'=\Q[[x,y,z]]/(x^2)$ is also the completion of an uncountable local domain with a countable spectrum, even though $T'$ is not even a domain itself. 
\end{example}

\subsection{Uncountable Excellent Domains with Countable Spectra}\label{subsec:excellent_uncountable}

An important class of rings that are of particular interest to algebraic geometers and number theorists are \textit{excellent} rings. For any $P\in\Spec(A)$, define $k(P)\coloneqq A_P/PA_P$.

\begin{definition}[\cite{rotthaus}, Definition 1.4]\label{defn:excellent}
A local ring $A$ is \textit{excellent} if
\begin{itemize}
    \item[(a)] for all $P\in\Spec(A)$, $\widehat{A}\otimes_A L$ is regular for every finite field extension $L$ of $k(P)$, and
    \item[(b)] $A$ is universally catenary.
\end{itemize}
\end{definition}

In \cite{loepp03}, Loepp characterizes completions of excellent domains in the characteristic zero case. 

\begin{theorem}[\cite{loepp03}, Theorem 9]\label{thm:loepp_excellent_domain_char}
Let $(T,M)$ be a complete local ring containing the integers. Then $T$ is the completion of a local excellent domain if and only if it is reduced, equidimensional, and no integer of $T$ is a zero divisor.
\end{theorem}

We characterize completions of uncountable excellent local domains with countable spectra, assuming that the completion contains the rationals. We make use of results from \cite{paper1} to accomplish this.
The following lemma identifies sufficient conditions on a local ring to be excellent.

\begin{lemma}[\cite{paper1}, Lemma 2.8]\label{lem:excellent_sufficient_criteria}
Let $(T,M)$ be a complete local ring that is equidimensional and suppose $\Q\subseteq T$. Given a subring $(A,A\cap M)$ of $T$ with $\widehat{A}=T$, $A$ is excellent if, for every $P\in\Spec(A)$ and for every $Q\in\Spec(T)$ with $Q\cap A=P$, $(T/PT)_Q$ is a regular local ring.
\end{lemma}

The following result characterizes completions of countable excellent local domains, assuming that the completion contains the rationals.

\begin{theorem}[\cite{paper1}, Theorem 3.10]\label{thm:countable_excellent_dimgeq1_char}
Let $(T,M)$ be a complete local ring with $\Q\subseteq T$. Then $T$ is the completion of a countable excellent local domain if and only if the following conditions hold:
\begin{enumerate}
    \item $T$ is equidimensional,
    \item $T$ is reduced, and
    \item $T/M$ is countable.
\end{enumerate}
\end{theorem}

Given that complete local rings are also excellent, we have the following corollary of Proposition~\ref{prop:dim0_precompletion_domain}.  In particular, Corollary~\ref{cor:dim0_precompletion_excellent_domain} characterizes completions of uncountable excellent local domains with countable spectra in the dimension zero case.

\begin{corollary}\label{cor:dim0_precompletion_excellent_domain}
Suppose $(T,M)$ is a complete local ring with $\dim T=0$. Then $T$ is the completion of an uncountable excellent local domain with a countable spectrum if and only if $T$ is an uncountable field.
\end{corollary}

\begin{proof}
Suppose $T$ is the completion of an uncountable excellent local domain with a countable spectrum. By Proposition~\ref{prop:dim0_precompletion_domain}, we have that $T$ is an uncountable field.

Suppose $T$ is an uncountable field. Then, $T$ is the completion of an uncountable excellent local domain with a countable spectrum, namely itself. 
\end{proof}

We now identify sufficient conditions for a complete local ring of dimension at least one and containing the rationals to be the completion of an uncountable excellent local domain with a countable spectrum. The following proposition is a modification of \cite[Theorem 6.4]{loepp2018uncountable}

\begin{proposition}\label{prop:b_excellent}
Suppose $(T,M)$ is a complete local ring such that $\dim T\geq 1$ and $\Q\subseteq T$, and suppose $T$ satisfies the following conditions:
\begin{enumerate}
    \item $T$ is reduced,
    \item $T$ is equidimensional, and
    \item $T/M$ is countable.
\end{enumerate}
Then $T$ is the completion of an uncountable excellent local domain with a countable spectrum.
\end{proposition}

\begin{proof}
Since $T$ satisfies conditions (1), (2), and (3), we have by Theorem~\ref{thm:countable_excellent_dimgeq1_char} that $T$ is the completion of a countable excellent local domain, $(S,S\cap M)$.  By Theorem~\ref{thm:B_completionT}, there exists an uncountable local domain $(B,B\cap M)$ such that $S\subseteq B\subseteq T$, $\widehat{B}=T$, and ideals of $B$ are extended from $S$. Finally, by Proposition~\ref{prop:B_countable_spec}, we have that $B$ has a countable spectrum. Thus, it remains to show that $B$ is excellent.

Since $\widehat{B}=T$ and $T$ is assumed to contain $\Q$ and be equidimensional, we can apply Lemma~\ref{lem:excellent_sufficient_criteria}.  In particular, it suffices to show that, for every $P\in\Spec(B)$ and for every $Q\in\Spec(T)$ satisfying $Q\cap B=P$, we have that $(T/PT)_Q$ is a regular local ring. Let $P\in\Spec(B)$, and suppose $Q\in\Spec(T)$ such that $Q\cap B=P$. Then,
\[Q\cap S=(Q\cap B)\cap S=P\cap S.\]
Since ideals of $B$ are extended from $S$, we can write $P$ as $P=(p_1,\ldots,p_k)B$ for $p_i\in S$. Then,
\[P\cap S=(p_1,\ldots,p_k)B\cap S\subseteq(p_1,\ldots,p_k)T\cap S=(p_1,\ldots,p_k)S.\]
Note that $(p_1,\ldots,p_k)S\subseteq P\cap S$, and so $(p_1,\ldots,p_k)B\cap S=(p_1,\ldots,p_k)S$. It follows that $Q\cap S=(p_1,\ldots,p_k)S$.

Now, $(T/(p_1,\ldots,p_k)T)_Q$ is a regular local ring, since $S$ is excellent and has completion $T$. But notice that $(p_1,\ldots,p_k)T=PT$, so we have that $(T/(p_1,\ldots,p_k)T)_Q=(T/PT)_Q$, and $(T/PT)_Q$ is a regular local ring as well. Thus, by Lemma~\ref{lem:excellent_sufficient_criteria}, $B$ is excellent.
\end{proof}

We are now able to prove another main result.

\begin{theorem}\label{thm:dim2+_excellent_precompletion_domain}
Suppose $(T,M)$ is a complete local ring with $\Q\subseteq T$.
\begin{itemize}
    \item If $\dim T=0$, then $T$ is the completion of an uncountable excellent local domain with a countable spectrum if and only if $T$ is an uncountable field.
    \item If $\dim T=1$, then $T$ is the completion of an uncountable excellent local domain with a countable spectrum if and only if 
\begin{enumerate}
    \item $T$ is reduced, and
    \item $T$ is equidimensional.
\end{enumerate}
    \item If $\dim T\geq 2$, then $T$ is the completion of an uncountable excellent local domain with a countable spectrum if and only if
\begin{enumerate}
    \item $T$ is reduced,
    \item $T$ is equidimensional, and
    \item $T/M$ is countable.
\end{enumerate}
\end{itemize}
\end{theorem}

\begin{proof}
If $\dim T=0$ then the desired statement is given by Corollary~\ref{cor:dim0_precompletion_excellent_domain}.

Let $(T,M)$ be a complete local ring with $\dim T\geq 1$, $\Q\subseteq T$, and suppose that $T$ satisfies conditions (1), (2), and (3). By Proposition~\ref{prop:b_excellent}, we have that $T$ is the completion of an uncountable excellent local domain with a countable spectrum. Now suppose $(T,M)$ is a complete local ring with $\dim T=1$, such that $T/M$ is uncountable and $T$ satisfies conditions (1) and (2). Since $T$ satisfies these conditions and $\Q\subseteq T$ (so every nonzero integer is a unit and not a zero divisor), $T$ is the completion of a local excellent domain $(A,A\cap M)$, by Theorem~\ref{thm:loepp_excellent_domain_char}. Note that $\dim A = 1$ and so $A$ has a countable spectrum.  Since the completion of $A$ is $T$, the map $A\to T/M^2$ is surjective (see, for example, Proposition 2.4 of \cite{paper1}). By Proposition~\ref{prop:t/m^2_countable}, we have that
\[\abs{A}\geq\abs{T/M^2}=\abs{T/M}.\]
Thus, since $T/M$ is uncountable, $A$ is uncountable as well.

Now suppose that $(T,M)$ is a complete local ring with $\dim T\geq 1$ and it is the completion of an uncountable excellent local domain with a countable spectrum. By Theorem~\ref{thm:loepp_excellent_domain_char}, it must be that $T$ is reduced and equidimensional. If $\dim T\geq 2$, then we must also have that $T/M$ is countable by Theorem~\ref{thm:small19_domain_char_countable_spec}.
\end{proof}

\begin{example}
If $T=\Q[[x,y,z]]/(xy)$, then, by Theorem~\ref{thm:dim2+_excellent_precompletion_domain}, we have that $T$ is the completion of an uncountable excellent local domain with a countable spectrum. 
\end{example}

We now have the following corollary in the case that $T$ is a UFD.

\begin{corollary}\label{cor:uncount_excellent_UFD_count_spec_char}
Suppose $(T,M)$ is a complete local UFD with $\Q\subseteq T$.
\begin{itemize}
    \item If $\dim T=0$, then $T$ is the completion of an uncountable excellent local UFD with a countable spectrum if and only if $T$ is uncountable.
    \item If $\dim T=1$, then $T$ is the completion of an uncountable excellent local UFD with a countable spectrum.
    \item If $\dim T\geq 2$, then $T$ is the completion of an uncountable excellent local UFD with a countable spectrum if and only if $T/M$ is countable.
\end{itemize}
\end{corollary}

\begin{proof}
First, if $\dim T=0$, then $T$ is a field since it is a domain. The result follows from the case $\dim T=0$ of Theorem \ref{thm:dim2+_excellent_precompletion_domain}.

Next, suppose that $(T,M)$ is a complete local UFD with $\dim T=1$ and $\Q\subseteq T$. Then, since $T$ is a domain, it is reduced and equidimensional. By Theorem~\ref{thm:dim2+_excellent_precompletion_domain}, $T$ is the completion of an uncountable excellent local domain with a countable spectrum, $(A,A\cap M)$. Since $T$ is a UFD, $A$ is as well. 

Suppose that $(T,M)$ is a complete local UFD with $\dim T\geq 2$, $\Q\subseteq T$, and $T/M$ is countable.  Since $T$ is a domain, it is reduced and equidimensional. By Theorem~\ref{thm:dim2+_excellent_precompletion_domain}, $T$ is the completion of an uncountable excellent local domain with a countable spectrum, $(A,A\cap M)$. Since $T$ is a UFD, $A$ is as well. If $(T,M)$ is the completion of an uncountable excellent local UFD with a countable spectrum $(A,A\cap M)$, then $T/M$ is countable by Theorem \ref{thm:dim2+_precompletion_domain}.
\end{proof}

\subsection{Uncountable UFDs with Countable Spectra}\label{subsec:ufd_uncountable}

In this subsection, we characterize completions of uncountable local UFDs with countable spectra. 
First, we provide results from \cite{heitmannUFD}, which identify necessary and sufficient conditions for a complete local ring to be the completion of a local UFD.

\begin{theorem}[\cite{heitmannUFD}, Theorem 1]\label{thm:heitmann_UFD_necc}
Let $R$ be an integrally closed local domain. Then no integer is a zero divisor in $\widehat{R}$. Moreover, $\widehat{R}$ is either a field, a DVR, or a ring with depth at least two.
\end{theorem}

\begin{theorem}[\cite{heitmannUFD}, Theorem 8]\label{thm:heitmann_UFD_suff}
Let $(T,M)$ be a complete local ring such that no integer is a zero divisor in $T$ and $\depth T\geq 2$. Then there exists a local UFD $A$ such that $\widehat{A}\cong T$ and $\abs{A}=\sup(\aleph_0,\abs{T/M})$. If $p\in M$ where $p$ is a nonzero prime integer, then $pA$ is a prime ideal.
\end{theorem}


Notice that Theorem~\ref{thm:heitmann_UFD_suff} provides sufficient conditions for a complete local ring to be the completion of a countable local UFD.

We now characterize completions of uncountable local UFDs with countable spectra in the cases of dimension zero and dimension one.

\begin{proposition}\label{prop:char_ufd_dim0_1}
Suppose $(T,M)$ is a complete local ring. 
\begin{itemize}
    \item If $\dim T=0$, then $T$ is the completion of an uncountable local UFD with a countable spectrum if and only if $T$ is an uncountable field.
    \item If $\dim T=1$, then $T$ is the completion of an uncountable local UFD with a countable spectrum if and only if $T$ is a DVR.
\end{itemize}
\end{proposition}

\begin{proof}
First suppose $\dim T=0$ and $T$ is the completion of an uncountable local UFD with a countable spectrum. By Theorem \ref{thm:dim2+_precompletion_domain}, $T$ is an uncountable field.
Now suppose $T$ is an uncountable field. Then, $T$ is an uncountable local UFD with one prime ideal, and $\widehat{T}=T$. Therefore, $T$ is indeed the completion of an uncountable local UFD with a countable spectrum.

Next, suppose $\dim T=1$ and $T$ is the completion of an uncountable local UFD with a countable spectrum $A$. Since $A$ is integrally closed, by Theorem~\ref{thm:heitmann_UFD_necc}, it must be that $\widehat{A}=T$ is a DVR (as it cannot be a field or have depth at least $2$). Now suppose that $T$ is a DVR. Then, $T$ itself is uncountable by Lemma~\ref{lem:t_mod_p_uncountable}. Since a DVR is also a UFD, $T$ is also a UFD. Finally, since $\widehat{T}=T$  and $T$ has two prime ideals, we have that $T$ is the completion of an uncountable UFD with a countable spectrum.
\end{proof}

In order to show that the rings we construct are UFDs, we make use of the following characterization of UFDs for Noetherian domains.

\begin{theorem}[\cite{matsumura}, Theorem 20.1]\label{thm:ufd_principal_char}
If $A$ is a Noetherian integral domain, then $A$ is a UFD if and only if every height $1$ prime ideal is principal.
\end{theorem}

We now use our construction from Section~\ref{sec:construction} to identify sufficient conditions in the case that the dimension of the rings in question are at least two. 

\begin{proposition}\label{prop:suff_ufd_dim_geq2}
Suppose $(T,M)$ is a complete local ring with $\dim T\geq 2$ and suppose $T$ satisfies the following conditions:
\begin{enumerate}
    \item no integer of $T$ is a zero divisor,
    \item the depth of $T$ is at least $2$, and
    \item $T/M$ is countable.
\end{enumerate}
Then $T$ is the completion of an uncountable local UFD with a countable spectrum.
\end{proposition}

\begin{proof}
Since $T$ satisfies conditions (1), (2), and (3), we have by Theorem~\ref{thm:heitmann_UFD_suff} that $T$ is the completion of a countable local UFD. Let this countable local UFD be $(S,S\cap M)$. By Theorem~\ref{thm:B_completionT}, there exists an uncountable local domain $(B,B\cap M)$ such that $S\subseteq B\subseteq T$, $\widehat{B}=T$, and ideals of $B$ are extended from $S$. Finally, by Proposition~\ref{prop:B_countable_spec}, we have that $B$ has a countable spectrum. Thus, it remains to show that $B$ is a UFD.

Let $P\in\Spec(B)$ be a height one prime ideal of $B$. By Theorem~\ref{thm:ufd_principal_char}, it is enough to show that $P$ is a principal ideal. By Proposition~\ref{prop:specb_specs_isomorphic}, $P$ corresponds to a height one prime ideal $Q$ of $S$. Since $S$ is a Noetherian UFD, we have by Theorem~\ref{thm:ufd_principal_char} that $Q$ is a principal ideal, so $Q=xS$ for some $x\in S$. By Remark~\ref{rem:BS_same_generators}, we have that $P$ is generated by the same elements that generate $Q$, so $P=xB$. Thus $P$ is principal, and so $B$ is indeed a UFD.
\end{proof}

We are now ready to characterize completions of uncountable local UFDs with countable spectra.

\begin{theorem}\label{thm:char_ufd_completions_dim_geq2}
Suppose $(T,M)$ is a complete local ring.
\begin{itemize}
    \item If $\dim T=0$, then $T$ is the completion of an uncountable local UFD with a countable spectrum if and only if $T$ is an uncountable field.
    \item If $\dim T=1$, then $T$ is the completion of an uncountable local UFD with a countable spectrum if and only if $T$ is a DVR.
    \item If $\dim T\geq 2$, then $T$ is the completion of an uncountable local UFD with a countable spectrum if and only if
\begin{enumerate}
    \item no integer of $T$ is a zero divisor,
    \item the depth of $T$ is at least $2$, and
    \item $T/M$ is countable.
\end{enumerate}
\end{itemize}
\end{theorem}

\begin{proof}
If $\dim T = 0$ or $\dim T = 1$, the desired statement follows from Proposition~\ref{prop:char_ufd_dim0_1}.

If $\dim T\geq 2$, by Proposition~\ref{prop:suff_ufd_dim_geq2}, conditions (1), (2), and (3) are sufficient for $T$ to be the completion of an uncountable local UFD with a countable spectrum. Now suppose that $T$ is the completion of such a UFD. By Theorem~\ref{thm:heitmann_UFD_necc}, we have that no integer of $T$ is a zero divisor and the depth of $T$ is at least $2$. By Theorem~\ref{thm:small19_domain_char_countable_spec}, we have that $T/M$ is countable.
\end{proof}

\begin{example}
Surprisingly, the ring $T'=\Q[[x,y,z]]/(x^2)$ from Example~\ref{eg:uncountable_domain_completion} is not only the completion of an uncountable local domain with a countable spectrum, but also the completion of an uncountable local UFD with a countable spectrum by Theorem~\ref{thm:char_ufd_completions_dim_geq2}.
\end{example}

\subsection{Uncountable Noncatenary Domains and UFDs with Countable Spectra}\label{subsec:noncat_uncountable}
In this subsection, we characterize completions of uncountable noncatenary local domains with countable spectra, and completions of uncountable noncatenary local UFDs with countable spectra. 

First, we provide results from \cite{SMALL17} and \cite{paper1}, which identify necessary and sufficient conditions for a complete local ring to be the completion of a noncatenary local domain and necessary and sufficient conditions for a complete local ring to be the completion of a countable noncatenary local domain.

\begin{theorem}[\cite{SMALL17}, Theorem 2.10]\label{thm:small17noncatdoman_char}
Let $(T,M)$ be a complete local ring. Then $T$ is the completion of a noncatenary local domain if and only if the following conditions hold:
\begin{enumerate}
    \item no integer of $T$ is a zero divisor,
    \item $M\notin\Ass(T)$, and
    \item there exists $P\in\Min(T)$ such that $1<\dim(T/P)<\dim T$.
\end{enumerate}
\end{theorem}

\begin{theorem}[\cite{paper1}, Theorem 4.5]\label{thm:countable_noncat_domain_char}
Let $(T,M)$ be a complete local ring. Then $T$ is the completion of a countable noncatenary local domain if and only if the following conditions hold:
\begin{enumerate}
    \item no integer of $T$ is a zero divisor,
    \item $M\notin\Ass(T)$,
    \item there exists $P\in\Min(T)$ such that $1<\dim(T/P)<\dim T$, and
    \item $T/M$ is countable.
\end{enumerate}
\end{theorem}

We use the latter result to identify sufficient conditions for a complete local ring to be the completion of an uncountable noncatenary local domain with a countable spectrum.

\begin{proposition}\label{prop:uncount_noncat_domain_suff}
Suppose $(T,M)$ is a complete local ring that satisfies the following conditions:
\begin{enumerate}
    \item no integer of $T$ is a zero divisor,
    \item $M\notin\Ass(T)$,
    \item there exists $P\in\Min(T)$ such that $1<\dim(T/P)<\dim T$, and
    \item $T/M$ is countable.
\end{enumerate}
Then $T$ is the completion of an uncountable noncatenary local domain with a countable spectrum.
\end{proposition}

\begin{proof}
Since $T$ satisfies conditions (1), (2), (3), and (4), we have by Theorem~\ref{thm:countable_noncat_domain_char} that $T$ is the completion of a countable noncatenary local domain, $(S,S\cap M)$. Notice that condition (3) guarantees that $\dim T\geq 1$. Thus, by Theorem~\ref{thm:B_completionT}, there exists an uncountable local domain $(B,B\cap M)$ such that $S\subseteq B\subseteq T$, $\widehat{B}=T$, and ideals of $B$ are extended from $S$. Finally, by Proposition~\ref{prop:B_countable_spec}, we have that $B$ has a countable spectrum. Thus, it remains to show that $B$ is noncatenary.

Since $S$ is noncatenary, there exists a pair of prime ideals $P\subsetneq Q$ of $S$ such that two saturated chains of prime ideals between $P$ and $Q$ have different lengths. By Proposition~\ref{prop:specb_specs_isomorphic}, $\Spec(B)$ and $\Spec(S)$ are order isomorphic, and so there also exist two prime ideals $P'\subsetneq Q'$ of $B$ such that two saturated chains of prime ideals between $P'$ and $Q'$ have different lengths. Thus, $B$ is noncatenary as well.
\end{proof}

We now characterize completions of uncountable noncatenary local domains with countable spectra.

\begin{theorem}\label{thm:uncount_noncat_domain_char}
Suppose $(T,M)$ is a complete local ring. Then $T$ is the completion of an uncountable noncatenary local domain with a countable spectrum if and only if the following conditions are satisfied:
\begin{enumerate}
    \item no integer of $T$ is a zero divisor,
    \item $M\notin\Ass(T)$,
    \item there exists $P\in\Min(T)$ such that $1<\dim(T/P)<\dim T$, and
    \item $T/M$ is countable.
\end{enumerate}
\end{theorem}

\begin{proof}
If $T$ satisfies conditions (1), (2), (3), and (4), then, by Proposition~\ref{prop:uncount_noncat_domain_suff}, $T$ is the completion of an uncountable noncatenary local domain with a countable spectrum. If $T$ is the completion of such a ring, then, by Theorem~\ref{thm:small17noncatdoman_char}, $T$ must satisfy conditions (1), (2), and (3). Since condition (3) is satisfied, we have that $\dim T>2$, so, by Theorem~\ref{thm:small19_domain_char_countable_spec}, it must be that $T/M$ is countable and condition (4) is satisfied.
\end{proof}

We now use similar arguments as those in Propositions~\ref{prop:suff_ufd_dim_geq2} and \ref{prop:uncount_noncat_domain_suff} to identify sufficient conditions for a complete local ring to be the completion of an uncountable noncatenary local UFD with a countable spectrum. The following results from \cite{SMALL17} and \cite{paper1} characterize completions of noncatenary local UFDs and countable noncatenary local UFDs.

\begin{theorem}[\cite{SMALL17}, Theorem 3.7]\label{thm:small17noncatufd_char}
Let $(T,M)$ be a complete local ring. Then $T$ is the completion of a noncatenary local UFD if and only if the following conditions hold:
\begin{enumerate}
    \item no integer of $T$ is a zero divisor,
    \item $\depth(T)>1$, and
    \item there exists $P\in\Min(T)$ such that $2<\dim(T/P)<\dim T$.
\end{enumerate}
\end{theorem}

\begin{theorem}[\cite{paper1}, Theorem 4.10]\label{thm:noncat_ufd_char}
Let $(T,M)$ be a complete local ring. Then $T$ is the completion of a countable noncatenary local UFD if and only if the following conditions hold:
\begin{enumerate}
    \item no integer of $T$ is a zero divisor,
    \item $\depth T>1$,
    \item there exists $P\in\Min(T)$ such that $2<\dim(T/P)<\dim T$, and
    \item $T/M$ is countable.
\end{enumerate}
\end{theorem}

We use Theorem~\ref{thm:noncat_ufd_char} to identify sufficient conditions on a complete local ring to be the completion of an uncountable noncatenary local UFD with a countable spectrum.

\begin{proposition}\label{prop:uncount_noncat_ufd_suff}
Suppose $(T,M)$ is a complete local ring that satisfies the following conditions:
\begin{enumerate}
    \item no integer of $T$ is a zero divisor,
    \item $\depth T>1$,
    \item there exists $P\in\Min(T)$ such that $2<\dim(T/P)<\dim T$, and
    \item $T/M$ is countable.
\end{enumerate}
Then $T$ is the completion of an uncountable noncatenary local UFD with a countable spectrum.
\end{proposition}

\begin{proof}
Since $T$ satisfies conditions (1), (2), (3), and (4), we have by Theorem~\ref{thm:noncat_ufd_char} that $T$ is the completion of a countable noncatenary local UFD, $(S,S\cap M)$. Notice that the conditions of Theorem~\ref{thm:B_completionT} are met, since $\dim T\geq 1$ by condition (3). Thus, there exists an uncountable local domain $(B,B\cap M)$ such that $S\subseteq B\subseteq T$, $\widehat{B}=T$, and ideals of $B$ are extended from $S$. Finally, by Proposition~\ref{prop:B_countable_spec}, we have that $B$ has a countable spectrum. Thus, it remains to show that $B$ is a noncatenary UFD.

By Proposition~\ref{prop:specb_specs_isomorphic}, $\Spec(B)$ and $\Spec(S)$ are order isomorphic. If $P\in\Spec(B)$ is a height one prime ideal, then it corresponds to a height one prime ideal of $S$, which is principal since $S$ is a Noetherian UFD. By Remark~\ref{rem:BS_same_generators}, this principal ideal of $S$ and $P$ are generated by the same elements of $S$, so $P$ is also principal. Thus, $B$ is a UFD. Similarly, since $S$ is noncatenary, there exist prime ideals $P\subsetneq Q$ of $S$ such that two saturated chains of prime ideals between $P$ and $Q$ have different lengths. Since $\Spec(B)$ and $\Spec(S)$ are order isomorphic, there exist corresponding saturated chains of different lengths between two prime ideals of $B$. Thus, $B$ is noncatenary.
\end{proof}

We now characterize completions of uncountable noncatenary local UFDs with countable spectra.

\begin{theorem}\label{thm:uncount_noncat_ufd_char}
Suppose $(T,M)$ is a complete local ring. Then $T$ is the completion of an uncountable noncatenary local UFD with a countable spectrum if and only if the following conditions are satisfied:
\begin{enumerate}
    \item no integer of $T$ is a zero divisor,
    \item $\depth T>1$,
    \item there exists $P\in\Min(T)$ such that $2<\dim(T/P)<\dim T$, and
    \item $T/M$ is countable.
\end{enumerate}
\end{theorem}

\begin{proof}
If $T$ satisfies conditions (1), (2), (3), and (4), then, by Proposition~\ref{prop:uncount_noncat_ufd_suff}, $T$ is the completion of an uncountable noncatenary local UFD with a countable spectrum. If $T$ is the completion of such a ring, then, by Theorem~\ref{thm:small17noncatufd_char}, $T$ must satisfy conditions (1), (2), and (3). Since condition (3) is satisfied, $\dim T>3$, so, by Theorem~\ref{thm:small19_domain_char_countable_spec}, we have that $T/M$ is countable and condition (4) must be satisfied as well.
\end{proof}

\begin{example}
By Theorem~\ref{thm:uncount_noncat_domain_char}, the ring $\Q[[x,y,z,w]]/(x)\cap(y,z)$ is the completion of an uncountable noncatenary local domain with a countable spectrum.
Similarly, by Theorem~\ref{thm:uncount_noncat_ufd_char}, the ring $\Q[[x,y_1,y_2,z_1,z_2]]/(x)\cap(y_1,y_2)$ is the completion of an uncountable noncatenary local UFD with a countable spectrum.
\end{example}

\section*{Acknowledgments}
We thank the Clare Boothe Luce Scholarship Program for supporting the research of the second author.

\end{document}